\documentclass[a4paper,leqno,11pt]{amsart}

\usepackage{amsfonts,amssymb,verbatim,amsmath,amsthm,latexsym,textcomp,amscd}
\usepackage{latexsym,amsfonts,amssymb,epsfig,verbatim}
\usepackage{amsmath,amsthm,amssymb,latexsym,graphics,textcomp}
\usepackage{graphicx}
\usepackage{color}
\usepackage{url}
\usepackage{enumerate}
\usepackage[mathscr]{euscript}

\input xy
\xyoption{all}

\usepackage{hyperref}

\usepackage{tikz-cd}

\theoremstyle{plain}
\newtheorem{theorem}{Theorem}[section]
\newtheorem{thm}[theorem]{Theorem}
\newtheorem{lemma}[theorem]{Lemma}
\newtheorem{cor}[theorem]{Corollary}
\newtheorem{prop}[theorem]{Proposition}

\newtheorem*{thma}{Theorem A}

\theoremstyle{definition}
\newtheorem{defn}[theorem]{Definition}
\newtheorem{rmk}[theorem]{Remark}
\newtheorem{exam}[theorem]{Example}
 \newtheorem{subsec}[theorem]{}
\newtheorem{notation}[theorem]{Notation}
\numberwithin{table}{section}

\newcommand*{\myTagFormat}[2]{(\ref{#1})($#2$)}

\newenvironment{myeq}[1][]
{\stepcounter{theorem}\begin{equation}\tag{\thetheorem}{#1}}
{\end{equation}}

\newenvironment{mysubsection}[2][]
{\begin{subsec}\begin{upshape}\begin{bfseries}{#2.}
\end{bfseries}{#1}}
{\end{upshape}\end{subsec}}

\newcommand{\Mack}{\sf Mack}
\newcommand{\Burn}{\sf Burn}
\newcommand{\Ab}{\sf Ab}

\newcommand{\ov}[1]{\overline{#1}}

\newcommand{\BB}{\mathcal{B}}

\newcommand{\Z}{\mathbb{Z}}
\newcommand{\Q}{\mathbb{Q}}

\newcommand{\uM}{\underline{M}}
\newcommand{\uN}{\underline{N}}

\newcommand{\uT}{\underline{T}}
\newcommand{\uA}{\underline{A}}
\newcommand{\uC}{\underline{C}}
\newcommand{\uR}{\underline{R}}

\newcommand{\uO}{\underline{0}}
\newcommand{\uZ}{\underline{\mathbb{Z}}}
\newcommand{\uZp}{\underline{\mathbb{Z}/p}}

\newcommand{\uB}{\underline{B}}

\newcommand{\Id}{\mathrm{Id}}
\newcommand{\I}{\operatorname{I}}

\newcommand{\Mod}{\mathrm{Mod}}

\newcommand{\Ext}{\operatorname{Ext}}
\newcommand{\op}{\operatorname{op}}
\newcommand{\colim}{\operatorname{colim}}
\newcommand{\im}{\operatorname{im}}
\newcommand{\Mor}{\operatorname{Mor}}

\newcommand{\res}{\operatorname{res}}
\newcommand{\tr}{\operatorname{tr}}

\newcommand{\Fun}{\operatorname{Fun}}

\newcommand{\Hom}{\operatorname{Hom}}
\newcommand{\uH}{\underline{H}}

\newcommand{\upi}{\underline{\pi}}
\newcommand{\uE}{\underline{E}}

\makeatletter
\def\thickhline{%
  \noalign{\ifnum0=`}\fi\hrule \@height \thickarrayrulewidth \futurelet
   \reserved@a\@xthickhline}
\def\@xthickhline{\ifx\reserved@a\thickhline
               \vskip\doublerulesep
               \vskip-\thickarrayrulewidth
             \fi
      \ifnum0=`{\fi}}
\makeatother

\newlength{\Oldarrayrulewidth}
\newcommand{\Cline}[2]{%
  \noalign{\global\setlength{\Oldarrayrulewidth}{\arrayrulewidth}}%
  \noalign{\global\setlength{\arrayrulewidth}{#1}}\cline{#2}%
  \noalign{\global\setlength{\arrayrulewidth}{\Oldarrayrulewidth}}}

\newlength{\thickarrayrulewidth}
\author{Samik Basu, Surojit Ghosh}

\email{samik.basu2@gmail.com; samikbasu@isical.ac.in}
\address{Stat-Math Unit,
Indian Statistical Institute,
B. T. Road, Kolkata-700108, INDIA.}

\email{surojitghosh89@gmail.com; surojit.ghosh@ma.iitr.ac.in}
\address{Department of Mathematics,
Indian Institute of Technology,
Roorkee-247667, INDIA} 


\begin{document}

\title[Non-trivial extensions in equivariant cohomology]{Non-trivial extensions in equivariant cohomology with constant coefficients}
\maketitle

\begin{abstract}
In this paper, we prove some computational results about equivariant cohomology over the cyclic group $C_{p^n}$ of prime power order. We show that there is an inductive formula when the dimension of the $C_p$-fixed points of the grading is large. Among other calculations, we also show the existence of non-trivial extensions when $n\geq 3$. 
\end{abstract}

\section{Introduction}
The equivariant stable homotopy category has a rich structure coming from the desuspension along representation spheres. This equips equivariant cohomology groups with a grading over the virtual representations $RO(G)$. The resulting structure is usually difficult to compute even in the case of ordinary cohomology. Their computations give interesting results and also have surprising consequences. 

The coefficients of ordinary cohomology are Mackey functors, and the important ones are the Burnside ring Mackey functor $\uA$, and the constant Mackey functor $\uZ$. For the group $C_p$, the $RO(C_p)$-graded commutative rings $\upi_{-\bigstar}^{C_p} H\uA \cong \uH^\bigstar_{C_p}(S^0;\uA)$ was computed by Lewis \cite{Lew88}, and analogous results for $\uZ$, $\uZp$ were computed by Stong and Lewis. There are computations for a few other groups (see \cite{BG19}, \cite{BG20}, \cite{Zen17} \cite{KL20}), however, for most ones, even among Abelian groups, very little is known. 

In this paper, we study the Mackey functors $\upi_{-\alpha}^{C_{p^n}} H\uZ \cong \uH^\alpha_{C_{p^n}}(S^0;\uZ)$ for $\alpha \in RO(C_{p^n})$, which form the additive structure for equivariant cohomology over the group $C_{p^n}$. We prove the following result. 
\begin{thma}
If $|\alpha^{C_p}|\leq -2n+2$ or $|\alpha^{C_p}|\geq 2n$, the Mackey functor $\uH^\alpha_{C_{p^n}}(S^0;\uZ)$ can be computed directly from the Mackey functor $\uH^{\alpha^{C_p}}_{C_{p^{n-1}}}(S^0;\uZ)$. (see Table \ref{comp-highfix})
\end{thma}
We also point out various computations of these Mackey functors not covered by the above theorem. The formulas that appear here are mostly written as a direct sum of those of the form $\uZ_T$ and $\uB_{T,S}$, a notation inspired from \cite{HHR17}. A consequence of these results are the complete calculation of the additive structure for the group $C_{p^2}$ (see Table \ref{comp-tab}). A new feature of these groups starting from $n\geq 2$ is the existence of $a_\lambda$-periodic classes. For the group $C_p$, the classes were either part of a polynomial algebra or were $a_\lambda$-torsion. 

In these computations, we also point out a non-trivial extension of Mackey functors. These extensions first occur in the case $n=3$, and hence also for higher $n$. In cohomology over the Burnside ring $\uA$, one has extensions $\uA[d]$ for the group $C_p$ of the form $0\to \langle \Z \rangle \to \uA[d] \to \uZ\to 0$, that occur in the additive structure. On the other hand, over $\uZ$ coefficients, the Mackey functors occuring in the additive structure over $C_p$ are a direct sum of those of the form $\uZ_T$ and $\uB_{T,S}$. For the group $C_{p^n}$, many special cases have been shown to be of this type (see for example \cite[Theorem 5.7]{HHR17}). However, we point out through examples that this does not happen in general, and in these cases, the $a_{\lambda_i}$-multiplication gives non-trivial extensions.   

\begin{mysubsection}{Organization}
In Section \ref{eqcoh}, we recall some preliminaries on equivariant cohomology, and their computational methods. In Section \ref{Z-mod}, we discuss the category of $\uZ$-modules, and their extensions, constructing important examples used in later sections. In Section \ref{largedimcomp}, we compute the equivariant cohomology at large $C_p$-fixed points.  In Section \ref{nontriv}, we compute the equivariant cohomology over $C_{p^2}$ and use them to point out the non-trivial extensions. 
\end{mysubsection}

\begin{notation}
Throughout this paper, $G$ denotes the cyclic group $C_{p^n}$ of order $p^n$, where $m$ is odd, and $g$ denotes a fixed generator of $G$. For an orthogonal $G$-representation $V$, $S(V)$ denotes the unit sphere, $D(V)$ the unit disk, and $S^V$ the one-point compactification $\cong D(V)/S(V)$. 
\end{notation}	

\section{Equivariant cohomology}\label{eqcoh}
Ordinary cohomology theories are defined for Abelian groups, and these are represented by spectra with homotopy concentrated in degree $0$. In the equivariant world, the analogous role is played by \emph{Mackey functors}. In this section we briefly recall their definition, and relate them to equivariant cohomology (see \cite{May96} for details). 

The \emph{Burnside category} $\Burn_G$ is the category with objects finite $G$-sets
 and each morphism set $\Mor_{\Burn_G}(S,T)$ is the group completion of the hom-set of spans between $S$ and $T$ in the category of finite $G$-sets. It is a fact that $\Burn_G$ is self dual, that is, the duality map $D\colon  \Burn_G \to \Burn_G^{\op}$ that is identity on objects and switches the legs of the spans, is an isomorphism of categories.
 
 \begin{defn}
A functor $\uM: \Burn_G^{op} \to \Ab$ from Burnside category into Abelian groups is called \emph{Mackey functor}.
\end{defn}

In this paper we restrict our attention to $G=C_{p^n}$. For the remainder of the paper $G$ will always refer to this group. 
Explicitly, a $G$-Mackey functor\footnote{This is a simplification in the case $G$ is Abelian. Otherwise the double coset formula (4) has a slightly more complicated expression.} $\uM$ is a collection of commutative $W_G(H)$-groups $\uM(G/H)$ one for each subgroup $H \le G$, each accompanied by \emph{transfer} $\tr^H_K\colon  \uM(G/K) \to \uM(G/H)$ and \emph{restriction} $\res^H_K \colon\uM(G/H) \to \uM(G/K)$ for $K\le H\le G$ such that
\begin{enumerate}
\item $\tr^H_J = \tr^H_J\tr^K_J$ and $\res^H_J = \res^K_J \res^H_K$ for all $J \le K \le H.$
\item  $\tr^H_K(\gamma.x)= \tr^H_K(x)$ for all $x \in \uM(G/K)$ and $\gamma \in W_H(K).$
\item $\gamma. \res^H_K(x) =\res^H_K(x)$ for all $x \in \uM(G/H)$ and $\gamma \in W_H(K).$
\item $\res^H_K\tr ^J_K(x)= \sum_{\gamma \in W_H(K)} \gamma.\tr^{K}_{J\cap K}(x)$
for all subgroups $J,H \leq K.$
\end{enumerate}
We will often write $\uM(H)$ for $\uM(G/H)$.

A morphism between two Mackey functors is given by natural transformations. We denote the category of Mackey functors of the group $G$, by $\Mack_G.$ It is a fact that $\Mack_G$ is an Abelian category. The Burnside ring Mackey functor is representable functor $\Burn_G(-, G/G)$. This is denoted by $\uA$. For an Abelian group $C$, the constant Mackey functor $\uC$ is described as $\uC(G/H)=C$ with the $\res^H_K=\Id$, and $\tr^H_K=$ multiplication by $[H:K]$. Following Lewis \cite{Lew88},  the data of a Mackey functor for the group $C_p$ may be organized in a diagram as demonstrated below. 
$$\xymatrix@R=0.1cm{    & \Z \oplus \Z \ar@/_.5pc/[dd]_{[\begin{smallmatrix} 1 & p \end{smallmatrix}]}   && & & \Z  \ar@/_.5pc/[dd]_{\Id}    \\  
  \uA :      &                                                                                                                              && &\uZ :   \\ 
  & \Z \ar@/_.5pc/[uu]_{\left[ \substack{0 \\ 1}\right]}  &&& & \Z \ar@/_.5pc/[uu]_{p}}$$
The top row gives the value of the Mackey functor at $C_p/C_p$, and the bottom row gives the Mackey functor at $C_p/e$. For the groups $C_{p^n}$, there are analogous diagrams arranged vertically with $n+1$ different levels.

\begin{exam}
For a $G$-spectrum $X$, the equivariant homotopy groups forms a Mackey functor $\upi_n(X)$, defined by the formula 
\[ \upi_n(X)(G/H):= \pi_n(X^H).
\]
\end{exam}

For a Mackey functor $\uM$ one may define an Eilenberg-MacLane spectrum \cite{GM95} $H\uM$ such that $\upi_n(H\uM)$ is concentrated in degree $0$ where it is the Mackey functor $\uM$. This constructs a $RO(G)$-graded cohomology theory by the formula 
\[
H^\alpha_G(X;\uM) \cong \mbox{ Ho-G-spectra } (X, \Sigma^\alpha H\uM).
\]
Recall that $RO(G)$ denotes the group completion of the monoid of irreducible representations of $G.$ A general element $\alpha \in RO(G)$ can be represented as a formal difference $\alpha = V - W$ for $G$-representations $V, W.$ For a unitary representation $V$ of $G$, we denote by $S^V$ the one-point compactification of $V$. Analogously for a virtual representation $\alpha= V-W$, $S^\alpha$ denotes the $G$-spectrum $\Sigma^{-W} S^V$. Using the functor $G/H\mapsto X\wedge G/H_+$, the cohomology groups are part of a Mackey functor denoted by $\uH_G^\alpha(S^0;\uM)$. 

One may put a symmetric monoidal structure on the category $\Mack_G$ by using the \emph{box product}. For two Mackey functors $\uM$ and $\uN \in \Mack_G$, the box product $\uM \square \uN$ is the left Kan extension of tensor product of Abelian groups along the functor $\times : \Burn_G^{\op} \times \Burn_G^{\op}\to \Burn_G^{\op}$ given by $(S, T) \mapsto S\times T.$ The Burnside ring Mackey functor $\uA$ plays the role of unit object in the symmetric monoidal structure of $\Mack_G.$

\begin{defn}
A ({\it commutative}) \emph{Green functor} for $G$,  is a (commutative) monoid in the symmetric monoidal category $\Mack_G$ defined as above. 
\end{defn}

Both $\uA$ and $\uZ$ are examples of commutative Green functors. Given a commutative Green functor $R$, an \emph{$\uR$-module} is a Mackey functor $\uM$ equipped with $\mu_{\uM} \colon \uR\square \uM\to \uM$ satisfying the usual relations.
 The category of $\uR$-modules will be denoted by $\uR$-$\Mod_G$. This in turn has the structure of a symmetric monoidal category with the induced box product. 
For a commutative Green functor $R$, the corresponding equivariant cohomology has a graded commutative ring structure. 

\begin{notation}
 The representation ring  $RO(C_{p^n})$ is generated by the trivial representation $1$, and the $2$-dimension representation $\lambda (k)$ given by the rotation by the angle $\frac{2 \pi k}{p^n}$ for $k=1, \cdots, \frac{p^n-1}{2}$. Denote the representation $\lambda (p^m)$ by $\lambda_m.$  Write $RO_0(C_{p^n}) \subset RO(C_{p^n})$ of those $\alpha$ such that the dimension of $\alpha$ is zero.
\end{notation}

We now describe some equivariant cohomology classes in $\uH^\bigstar_G(S^0;\uZ)(G/G) \cong \pi_{-\bigstar}^G(H\uZ)$. The generators used are defined  in \cite [Section 3]{HHR16} which we now recall. 
\begin{defn}\label{uagen}
Let $V$ be a $G$-representation. We have the $G$-map $S^0 \to S^V$\footnote{This is given by the inclusion of $\{0,\infty\}\subseteq S^V$.} which induces 
$$S^0 \to S^V\wedge S^0 \to S^V \wedge H\uZ$$ 
which we call $a_V\in \uH_G^V(S^0;\uZ)$. If $V$ is an oriented $G$-representation, a choice of orientation gives a class
$$u_V\in H_G^{V-\dim(V)}(S^0;\uZ) \cong \Z.$$
\end{defn}
We also have relations among these generators namely $u_V u_W = u_{V\oplus W}$ and $a_V a_W = a_{V\oplus W}$. It follows that these classes are products of $u_{\lambda(m)}$ and $a_{\lambda(m)}$ for $0\leq m <n$. These classes satisfy relations 
\begin{myeq}\label{a-reln}
p^{n-i} a_{\lambda_i}=0,
\end{myeq}  
and
\begin{myeq}\label{au-reln}
u_{\lambda_i} a_{\lambda_j} = p^{j-i} u_{\lambda_j} a_{\lambda_i} \mbox{   if } i < j.
\end{myeq}

\begin{defn}\label{ratio-classes}
Observe that the map $a:S^{\lambda_i} \to S^{\lambda_j}$ for $i<j<n$, described as the map on one-point compactifications induced by $z\mapsto z^{p^{j-i}}$, satisfies $a\circ a_{\lambda_i}  = a_{\lambda_j}$. This class is denoted by $a_{\lambda_j/\lambda_i}$. In the case $j=n$, this construction also makes sense and gives a class which after multiplication with $u_{\lambda_i}$ gives $p^{n-i}$, so it is denoted $[p^{n-i}u_{\lambda_i}^{-1}]$. 
\end{defn}

In this notation, for $G=C_p$, the equivariant cohomology of $S^0$ in gradings $n+m\lambda_0$ is given by (see \cite{Lew88}, \cite{Zen17}) 
\[
\uH^{\cdot + \cdot \lambda_0}_{C_p}(S^0;\uZ)(C_p/C_p) \cong \Z[u_{\lambda_0},a_{\lambda_0}] \oplus_{j\geq 1} \Z\{[pu_{\lambda_0}^{-j}]\} \oplus_{j,k\geq 1} \Sigma^{-1} \Z/p\{u_{\lambda_0}^{-j} a_{\lambda_0}^{-k}\}. 
\]
For $G=C_{p^n}$, the ring $\uH^{\bigstar}_{C_{p^n}}(S^0;\uZ)(C_{p^n}/C_{p^n})$ in gradings which are combinations of integers and positive multiples of $\lambda_i$ is described as \cite[Remark 4.6]{HHR17}
\[ \Z[a_{\lambda_0},u_{\lambda_0}, \cdots, a_{\lambda_{n-1}}, u_{\lambda_{n-1}}]/(p^{n-i}a_{\lambda_i}, u_{\lambda_i}a_{\lambda_{i+k}}-p^k u_{\lambda_{i+k}}a_{\lambda_i}, i,k\geq 0).
\] 

\vspace{0.2cm}

\begin{notation} Recall \cite{Web00} for $H\le G$, there is the \emph{restriction} functor 
\[\downarrow^G_H : \Mack_G \to \Mack_H\] given by $\downarrow^G_H(\uN)(H/L): = \uN(G\times_H H/L)$ where $\uN \in \Mack_G$ and $L \le H.$  Given a Mackey functor $\uM$, one defines $\Hom_L(\uM, \Z)$ as the composition of the functors 
\[ 
\xymatrix{\Burn_G^{\op} \ar[r]^D & \Burn_G^{\op} \ar[r]^{\uM} & \Ab \ar[rr]^{\Hom_{\Z}(-, \Z)} && \Ab}
\]
and similarly $\Ext_L(\uM, \Z)$ as the composition of the functors
\[
\xymatrix{\Burn_G^{\op} \ar[r]^D & \Burn_G^{\op} \ar[r]^{\uM} & \Ab \ar[rr]^{\Ext^1_{\Z}(-, \Z)} && \Ab}.
\]
We often denote $\Ext_L(\uM, \Z)$ by $\uM^E.$ The Mackey functor $\Hom_L(\uZ,\Z)$ has the same groups as $\uZ$ with the restriction and transfer maps switched and is denoted by $\uZ^\ast$. 
\end{notation}
 These Mackey functors play crucial role in the equivariant analog of the universal coefficient theorem discussed below along the lines of \cite{And69}.

\begin{mysubsection}{Anderson duality}
Let $\I_{\Q}$ and $\I_{\Q/\Z}$ be the spectra representing the cohomology theories given by $X \mapsto \Hom(\pi_{-\ast}^G(X), \Q)$ and $X \mapsto \Hom(\pi_{-\ast}^G(X), \Q/\Z)$ respectively. The natural map $\Q \to \Q/\Z$ induces the spectrum map $\I_{\Q} \to \I_{\Q/\Z}$, and the homotopy fibre is denoted by $\I_{\Z}$. For a $G$-spectrum $X$, the \emph{Anderson dual} $\I_{\Z}X$ of $X$, is the function spectrum $\Fun(X, \I_{\Z})$. For $X=H\uZ$, one easily computes $\I_{\Z}H\uZ \cong \Sigma^{2-\lambda_0}H\uZ.$

In general, for $G$-spectra $E$, $X$, and $\alpha \in RO(G),$ there is short exact sequence  
\begin{myeq}\label{end_dual}
0 \to \Ext_L(\uE_{\alpha -1}(X), \Z) \to \I_{\Z}(E)^{\alpha}(X) \to \Hom_L(\uE_{\alpha}(X), \Z)\to 0.
\end{myeq}
 In particular, for $E=H\Z$ and $X=S^0,$ we have the equivalence $\uE_\alpha(X) \cong \uH_G^{-\alpha}(S^0; \uZ)$. Therefore, one may rewrite \eqref{end_dual} as
 \begin{myeq}\label{and_comp}
 0 \to  \Ext_L(\uH^{3-\lambda_0 -\alpha}_{G}(S^0; \uZ), \Z) \to \uH^{\alpha}_{G}(S^0; \uZ) \to \Hom_L(\uH^{2-\lambda_0-\alpha}_{G}(S^0; \uZ), \Z) \to 0
 \end{myeq}
 for each $\alpha \in RO(G).$ 
\end{mysubsection}

Anderson duality provides a relation in the equivariant cohomology ring of $S^0$. A naive method to give another such relation is to build up $S^V$ using a filtration such that the filtration quotients are computable, and then use this to relate $\uH^\alpha_G(S^0)$ to $\uH^{\alpha - V}_G(S^0)$. This method is commonly used (see for example \cite{Lew88}, \cite{BG19}, \cite{BG20}). More explicitly, for each $m \le n$, we have homotopy cofibration sequences in $C_{p^n}$-spectra,
\begin{myeq}\label{rep_sphere}
{C_{p^n}/C_{p^m}}_+ \stackrel{1-g}{\to} {C_{p^n}/C_{p^m}}_+  \to S(\lambda_m)_+
\end{myeq}
and 
\begin{myeq}\label{sphere}
S(\lambda_m)_+ \to S^0 \to S^{\lambda_m}.
\end{myeq}
Here $g$ is a chosen generator for the quotient group ${C_{p^n}/C_{p^m}}.$

For a non-negative integer $0 \le k \le n,$ and a Mackey functor $\uM \in \Mack_{C_{p^k}}$ define 
$
\Theta_k : \Mack_{C_{p^k}} \to \Mack_{C_{p^n}}
$
as
\[\Theta_k(\uM)(G/H)= \begin{cases} \uZ(G/H)\otimes_\Z\uM(C_{p^k}/C_{p^k}) & \mbox{ if } C_{p^k} \subseteq H \\ \uM(C_{p^k}/H) & \mbox{ otherwise.}\end{cases} \]
The restrictions and transfers are clear from this description. In similar fashion, we define another functor $\Theta^\ast_k: \Mack_{C_{p^k}} \to \Mack_{C_{p^n}}$ as
\[\Theta^\ast_k(\uM)(G/H)= \begin{cases} \uZ^\ast(G/H)\otimes_\Z\uM(C_{p^k}/C_{p^k}) & \mbox{ if } C_{p^k} \subseteq H \\ \uM(C_{p^k}/H) & \mbox{ otherwise.}\end{cases} \]

\begin{prop}\label{sph_coh}
For each non-negative integer $m \le n$, $\uM \in \Mack_{C_{p^n}},$ and $\alpha \in RO(C_{p^n})$, there is short exact sequence
\[ \uO \to \Theta^\ast_m(\uH^{\alpha-1}_{C_{p^m}}(S^0; \downarrow^{p^n}_{{p^m}}\uM)) \to \uH^\alpha_{C_{p^n}}(S(\lambda_m)_+; \uM) \to \Theta_m(\uH^\alpha_{C_{p^m}}(S^0; \downarrow^{p^n}_{{p^m}}\uM))\to \uO\]
in $\Mack_{C_{p^n}}.$
\end{prop}
\begin{proof}
The cofiber sequence \eqref{rep_sphere} yields the cohomology long exact sequence
\[
\xymatrix@C=0.5cm{\cdots \uH^{\alpha-1}_{C_{p^n}}({C_{p^n}/C_{p^m}}_+) \ar[r] & \uH^\alpha_{C_{p^n}}(S(\lambda_m)_+) \ar[r] & \uH^{\alpha}_{C_{p^n}}({C_{p^n}/C_{p^m}}_+) \ar[r]^{(1-g)^\ast} & \uH^{\alpha}_{C_{p^n}}({C_{p^n}/C_{p^m}}_+)\cdots}
\]
An immediate computation gives $\ker((1-g)^\ast) \cong \Theta_m(\uH^\alpha_{C_{p^m}}(S^0; \downarrow^{p^n}_{{p^m}}\uM))$ and the cokernel of $(1-g)^\ast$ is $\Theta^\ast_m(\uH^{\alpha-1}_{C_{p^m}}(S^0; \downarrow^{p^n}_{{p^m}}\uM)).$ (See Proposition 4.3. and \S 4.5 in \cite{BG20} for an analogous computation.) Hence the result follows.
\end{proof}

\begin{exam} \label{lambda0}
Observe that $\Theta_0 (C)$ for an Abelian group $C$ is the constant Mackey functor $\uC$, while $\Theta_0^\ast(C)$ is the dual $\uC^\ast$. In Proposition \ref{sph_coh}, $\uH^\alpha_{e}(S^0;\uZ)$ is $0$ for $|\alpha|\neq 0$, and $\Z$ for $|\alpha|=0$. It follows that 
\begin{myeq}\label{lambda0coh}
\uH^{\alpha}_{C_{p^n}}(S(\lambda_0)_+;\uZ) \cong \begin{cases} \uZ & \mbox{ if } |\alpha|=0 \\ 
                                   \uZ^\ast &\mbox{ if } |\alpha|=1 \\ 
                                   0  &\mbox{  otherwise}. \end{cases}  
\end{myeq}

For each $m< n$, consider the following long exact sequence in cohomology associated to \eqref{sphere}
\refstepcounter{theorem}\label{eq:base}%
\begin{equation}
\xymatrix@C=0.6cm{\cdots \uH^{\alpha+\lambda_m-1}_{C_{p^n}}(S({\lambda_m}_+)) \ar[r] & \uH^{\alpha}_{C_{p^n}}(S^0) \ar[r]^{\cdot a_{\lambda_m}} & \uH^{\alpha+\lambda_m}_{C_{p^n}}(S^0) \ar[r] & \uH^{\alpha+\lambda_m}_{C_{p^n}}(S({\lambda_m}_+)) \cdots}\tag*{\myTagFormat{eq:base}{m}}\label{eq:base-m}
\end{equation}
At $m=0$, \eqref{lambda0coh} implies that multiplication by $a_{\lambda_0}$ is an isomorphism unless $|\alpha|=-2,-1,0$. At $|\alpha|=-2$, multiplication by $a_{\lambda_0}$ is injective while at $|\alpha|=-1,0$, multiplication by $a_{\lambda_0}$ is surjective. The injectivity at $|\alpha|=-1$ requires some additional argument which is essentially the same as \cite[Proposition 4.5]{BG20}.
\end{exam}


 \section{$\uZ$-modules \& extensions involving them}\label{Z-mod}
In this section, we discuss the category of $\uZ$-modules and a few particular $\Z$-modules that are important in the rest of  the paper. Along the way we study certain extensions in the additive category $\uZ$-$\Mod_{G}.$ We first note that the $\uZ$-modules satisfy an additional condition on the restriction and transfer maps. 
\begin{rmk}\label{cohomological}
For  any $\uM\in \uZ$-$\Mod_G$, $\tr^H_K \res^H_K$ equals the multiplication by index $[H: K]$ for $K \le H \le G$ \cite[Theorem 4.3]{Yos83}.
\end{rmk}

This condition puts certain restriction on which Mackey functors may actually be $\uZ$-modules. One such example is the lemma below.
\begin{lemma}\label{torsion}
If $\uM$ is a $\uZ$-$\Mod_G$ satisfying $\uM(G/e)=0$, then $\uM(G/H)\otimes_{\Z} \Q=0$ for all $H \le G.$
\end{lemma}
\begin{proof}
Applying Remark \ref{cohomological} to each $x \in \uM(G/H)$, $|H|x= \tr^H_e \res^H_e(x)= \tr^H_e(0)=0.$ This implies each element in $\uM(G/H)$ is torsion.
\end{proof}

We now interpret Lemma \ref{torsion} for equivariant cohomology with $\uZ$-coefficients. 
\begin{cor}\label{cortors}
Let $\alpha \in RO(G)\setminus RO_0(G)$, then the Mackey functor $\uH^\alpha_G(S^0; \uZ)$ is torsion.
\end{cor}
\begin{proof}
Note that the Mackey functors $\uH^\alpha_G(S^0;\uZ)$ are all $\uZ$-modules. For $\alpha$ with $|\alpha|\neq 0,$ one sees $\uH^{\alpha}_G(S^0; \uM)(G/e)\cong \tilde{H}^{|\alpha|}(S^0; \uM(G/e))$, hence is zero. By Lemma \ref{torsion}, the result follows.
\end{proof}

We now observe that $\uM \in \uZ$-$\Mod_G$ satisfying $\uM(G/e)=0$, then $\Hom_L(\uM, \Z)=0$. Applying this fact to  Corollary \ref{cortors} using \eqref{and_comp}, we note
\begin{myeq}\label{anders-comp}
|\alpha|\neq 0 \implies \uH^\alpha_{C_{p^n}}(S^0;\uZ) \cong \Ext_L(\uH^{3-\lambda_0 -\alpha}_{C_{p^n}}(S^0; \uZ), \Z).
\end{myeq}

The following proposition allows us to reduce the grading from $RO(C_{p^n})$ to the linear combinations of $\lambda_k$. 
\begin{prop}[\cite{Zen17}, Proposition 4.25]
There is an equivalence $H\uZ\wedge S^{\lambda(p^k)}\simeq H\uZ \wedge S^{\lambda(rp^k)}$ whenever $p \nmid r$. 
\end{prop}
The above implies the existence of invertible classes in $H\uZ$-cohomology in grading $\lambda(p^k)-\lambda(rp^k)$ for $p\nmid r$. One may make a coherent choice of these units, so that the $H\uZ$-cohomology is, up to some units and their inverses, the part which lies in the graded pieces given by linear combinations of $1,\lambda_0, \cdots, \lambda_{n-1}$. From now onwards, we assume that $\alpha\in RO(G)$ means that $\alpha = c + \sum_{k\geq 0} a_k \lambda_k$. 
\begin{defn}
Let $S\subseteq \overline{n}:=\{1,\cdots, n\}.$
 Denote by $\uZ_S$ the Mackey functor 
\[ \uZ_S(C_{p^n}/H)= \Z \mbox{ for } H\le C_{p^n},\] 
\[ \res^{C_{p^i}}_{C_{p^{i-1}}}= p^{\chi_S(i)}\text{ for } 1\le i \le n.\]
Here $\chi_S$ is the characteristic function on $S$.  
 \end{defn}
  
We note that $\uZ_{\emptyset}=\uZ$, and $\uZ_{\overline{n}}=\uZ^\ast$.  For subsets $S \subseteq T \subseteq \overline{n}$, there is a unique map $f_{T, S}: \uZ_T \to \uZ_S$ such that it is the identity at $C_{p^n}/e$\footnote{If $S \nsubseteq T$,  there is no Mackey functor morphism from $\uZ_T \to \uZ_S$ which induces identity at $C_{p^n}/e.$}. Then the Mackey functor structure suggests $f_{T,S}(C_{p^n}/{C_{p^k}})$ is the multiplication by $p^{\alpha_{T\setminus S,k}}$, where $\alpha_{T \setminus S,k}:=\#((T\setminus S)\cap \overline{k}).$ Denoting the cokernel of $f_{T,S}$ by  $\uB_{T,S}$, we get a exact sequence
  \begin{myeq}\label{nz_ext1}
 \uO\to \uZ_T \to \uZ_S\to \uB_{T, S}\to \uO.
 \end{myeq}
 where the Mackey functor $\uB_{T,S}$ is given by
  \[ 
 \uB_{T, S}(C_{p^n}/C_{p^k})=
\frac{\Z}{{p^{\alpha_{{T \setminus S} ,k}}\Z}}, ~~
 \res^{C_{p^k}}_{C_{p^{k-1}}}= 
 \begin{cases} 
 p & \text{ if } k \in S \\
 1 & \text{ if } k \in S^c,~~
  \end{cases}
 \tr^{C_{p^k}}_{C_{p^{k-1}}}= 
 \begin{cases} 
 1 & \text{ if } k \in S \\
  p & \text{ if } k \in S^c.
  \end{cases}
 \]
 %

 %
 
 %
For $k\leq \min(S \cup T)$, we readily observe that $\uB_{T,S} = \uB_{T\cup \ov{k}, S\cup \ov{k}}$. Applying the functor $\Hom_L(-, \Z)$ to  the short exact sequence \eqref{nz_ext1} yields the long exact sequence
 \[ \Hom_L(\uB_{T, S}, \Z) \to \Hom_L(\uZ_S, \Z) \to \Hom_L(\uZ_T, \Z)\to \Ext_L(\uB_{T, S}, \Z)\to \Ext_L(\uZ_S, \Z)\cdots\]
 One readily observes that the first and last terms of the long exact sequence above are zero, which simplifies the expression to the short exact sequence
 \[\uO \to \uZ_{S^c} \to \uZ_{T^c} \to \Ext_L(\uB_{T, S}, \Z)\to \uO.
 \]
 Comparing the short exact sequence above with \eqref{nz_ext1}, we deduce $\Ext_L(\uB_{T, S}, \Z) = \uB_{S^c, T^c}$.

 \begin{prop}\label{zb}
 Let $\uM$ be a torsion free $\uZ$-module for the group $C_{p^n}$  that fits into the short sequence 
 \[\uO \to \uZ_S \to \uM \to \uB_{S,T} \to \uO.
 \]
 Then, there is an isomorphism $\uM \cong \uZ_T$ of $\uZ$-modules.
 \end{prop}
 \begin{proof}
 Since $\uM$ is torsion free, that is, for each $0 \le k \le n,$ the Abelian group $\uM(C_{p^n}/C_{p^k})$ has no torsion. Thus $\uM(C_{p^n}/C_{p^k})\cong \Z.$ Therefore, there is a subset $T' \subseteq \ov{n}$ such that $\uM \cong \uZ_{T'},$ and then the cokernel of $\uZ_S \to \uM$ is $\uB_{S,T'}$. However from the formula of $\uB_{S,T}$ we easily observe that, $\uB_{S,T'} \cong \uB_{S,T}$ implies $T=T'$.
 \end{proof}
 
We demonstrate an exact sequence as in Proposition \ref{zb} using Mackey functor diagrams for the group $C_{p^3}$.
 \[
\xymatrix{&\Z  \ar@/_1pc/[d]_{p} \ar[rrrr]^-{p^3}& &&& \Z  \ar@/_1pc/[d]_{1} \ar[rrrr] & & & & \Z/p^3  \ar@/_1pc/[d]_{1}  \\
\uZ^\ast : &\Z \ar@/_1pc/[d]_{p} \ar@/_1pc/[u]_{1} \ar@/^1pc/[rrrr]^-{p^2}&&& \uZ: & \Z  \ar@/_1pc/[d]_{1} \ar@/_1pc/[u]_{p} \ar@/^1pc/[rrrr] & & &\uB_{\ov{3}, \emptyset}: & \Z/p^2  \ar@/_1pc/[d]_{1} \ar@/_1pc/[u]_{p}  \\
  &\Z \ar@/_1pc/[d]_{p} \ar@/_1pc/[u]_{1} \ar[rrrr]^-{p} &&& & \Z \ar@/_1pc/[d]_{1} \ar@/_1pc/[u]_{p} \ar[rrrr] & & & & \Z/p \ar@/_1pc/[d]_{1} \ar@/_1pc/[u]_{p} \\ 
  &\Z \ar@/_1pc/[u]_{1} \ar[rrrr]^{1} & &&&  \Z  \ar@/_1pc/[u]_{p} \ar[rrrr] & && &   0  \ar@/_1pc/[u]_{p}}
\]
 \begin{prop}\label{Bker} Let $k\le n$. \\
1) A map $f: \uB_{\ov{n}, \ov{k}^c} \to \uB_{\ov{1}, \emptyset}$ is uniquely determined by $f(C_{p^n}/C_p)$. \\
2) If $f(C_{p^n}/C_p)$ is an isomorphism, then $\mbox{Ker}(f) \cong \uB_{\ov{1}^c, \ov{1+k}^c}$, and $\mbox{Coker}(f) \cong \uB_{\{k+1\}, \emptyset}$. 
 \end{prop}
 
 \begin{proof}
The proof relies on a careful examination of the restriction and transfer maps of $\uB_{\ov{n},\ov{k}^c}$. These are given by 
\[
\uB_{\ov{n}, \ov{k}^c}(C_{p^n}/C_{p^r})=
\frac{\Z}{{p^{\min(r,k)}\Z}}, ~~
\res^{C_{p^r}}_{C_{p^{r-1}}}= 
 \begin{cases} 
 1 & \text{ if } r \leq k \\
 p & \text{ if } r>k,~~
  \end{cases}
 \tr^{C_{p^r}}_{C_{p^{r-1}}}= 
 \begin{cases} 
 p & \text{ if } r\leq k \\
 1 & \text{ if } r>k.
  \end{cases}
 \]
 The unique extension of $f$ from level $C_{p^n}/C_p$ is guaranteed by the fact that the restriction maps in $\uB_{\ov{1},\emptyset}$ are the identity above this level. If $f(C_{p^n}/C_p)$ is an isomorphism, we have
 \[
 f(C_{p^n}/C_{p^l}) \text{ is } \begin{cases}
 \text{onto} & \text{ if } l\le k \\ 
 0 & \text{ if } l \ge k+1.
 \end{cases}
 \]
This implies the required conclusion about the cokernel of $f$. Note that the part of $f$ between the levels $C_{p^n}/C_{p^k}$ and $C_{p^n}/C_{p^{k+1}}$ may be described as 
 
 \[
\xymatrix@C=0.7cm@R=0.3cm{\ker(f)  && \uB_{\ov{n}, \ov{k}^c}  & &  \uB_{\ov{1}, \emptyset} \\ 
\Z/p^k \ar@/_/[dd]_1 \ar[rr]^{\cong} &&  \Z/p^k \ar@/_/[dd]_p \ar[rr]^{0} & &\Z/p  \ar@/_/[dd]_1\\ 
& & &  && \\
 \Z/p^{k-1}\{ p\} \ar@/_/[uu]_p \ar[rr] && \Z/p^k \ar@/_/[uu]_1 \ar[rr]^{f(C_{p^n}/C_{p^{k}})} && \Z/p \ar@/_/[uu]_0}
 \]
Therefore, $\downarrow^{p^n}_{p^{k+1}}\ker(f) \cong \uB_{\ov{1}^c,\emptyset}.$ The part of this Mackey functor above the level $C_{p^n}/C_{p^{k+1}}$ is unchanged, that is, same with $\uB_{\ov{n}, \ov{k}^c}.$ Hence the result follows.
 \end{proof}
 
 \begin{mysubsection}{Pullback Mackey functors} 
The $\uZ$-modules may also be defined as Abelian group-valued functors $\uM: \BB\uZ_G^{\op}\to \Ab$. The category $\BB\uZ_G$ has finite $G$-sets as objects and $\Mor_{\BB\uZ_G}(S,T):= \Mor_G(\Z[S], \Z[T])$ (see \cite[Proposition 2.15]{Zen17}). Suppose $N$ is a normal subgroup of $G$. The quotient map $ G \to G/N$ induces $\phi: \BB\uZ_{G/N}^{\op} \to \BB\uZ_{G}^{\op}$. Define $\Phi^\ast_{N}: \uZ\text{-}\Mod_{G/N}\to \uZ\text{-}\Mod_{G}$ as $\Phi^\ast_{N}(\uM):= Lan_{\phi}(\uM)$, the left Kan extension of $\uM$ along $\phi$. For $G=C_{p^n}$ we write, $\Phi^\ast_{p^m}$ for $\Phi^\ast_{C_{p^m}}$. The Mackey functor $\Phi^\ast_N\uM$ is given by the formula
\begin{myeq}\label{Phi_Comp}\Phi^\ast_N\uM(G/H): = \underset{(G/H \to G/L) \in \Mor_{\BB_G}}\colim \uM((G/N)/(L/N))=\begin{cases} \uM((G/N)/(H/N)) & \text{ if } N \subseteq H \\ \uM((G/N)/(N/N)) & \text{ if } H \subset N.\end{cases}
\end{myeq}
and the restriction $\res^N_e$ is $\Id.$ From this formula, we note that $\Phi_p^\ast$ commutes with $\Ext_L$ on the $\Z$-modules which are $0$ at $C_{p^n}/e$. 
\begin{exam}
The formula above implies $\Phi^\ast_N\uZ\cong \uZ$. Also, $\Phi_{p^m}^\ast (\uB_{T,S})\cong \uB_{T^{(m)}, S^{(m)}}$, where $T^{(m)}$ is the image of $T$ under the map $\ov{n-m} \to \ov{n}$ given by $r \mapsto r+m$. 
\end{exam}
\end{mysubsection}

We now point out a non-trivial extension of $\uZ$-modules that arise in equivariant cohomology computations over $C_{p^2}$. 
\[
\xymatrix{&\Z/p  \ar@/_1pc/[d]_{0} \ar[rrrr] & &&& \Z \oplus \Z/p \ar@/_1pc/[d]_{{\begin{bmatrix} 1 & 0 \\ 1 & 0 \end{bmatrix}}} \ar[rrrr] & & & & \Z  \ar@/_1pc/[d]_{1}  \\
\uB_{\ov{2},\{2\}} : &\Z/p \ar@/_1pc/[d]_{0} \ar@/_1pc/[u]_{1} \ar@/^1pc/[rrrr] &&& \uT(2): & \Z \oplus \Z/p \ar@/_1pc/[d]_{{\begin{bmatrix} p & 0 \end{bmatrix}}} \ar@/_1pc/[u]_{{\begin{bmatrix} p & 0 \\ -1 & 1 \end{bmatrix}}} \ar@/^1pc/[rrrr] & & &\uZ_{\{1\}}: & \Z  \ar@/_1pc/[d]_{p} \ar@/_1pc/[u]_{p}  \\
  &0  \ar@/_1pc/[u] \ar[rrrr] &&& & \Z \ar@/_1pc/[u]_{{\begin{bmatrix} 1 \\ 0 \end{bmatrix}}} \ar[rrrr] & & & & \Z  \ar@/_1pc/[u]_{1}  }
\]
 A change of basis gives another isomorphic form of $\uT(2)$ occuring in the diagram above. 
\[
\xymatrix{  \Z \oplus \Z/p \ar@/_1pc/[d]_{{\begin{bmatrix} 1 & 0 \\ 0 & 0 \end{bmatrix}}} \\
 \Z \oplus \Z/p \ar@/_1pc/[d]_{{\begin{bmatrix} p & 0 \end{bmatrix}}} \ar@/_1pc/[u]_{{\begin{bmatrix} p & 0 \\ 0 & 1 \end{bmatrix}}}   \\
  \Z \ar@/_1pc/[u]_{{\begin{bmatrix} 1 \\ -1 \end{bmatrix}}} }
\]
This generalizes to $C_{p^n}$ Mackey functors $\uT(n)$ obtained by repeating top part of the diagram above. This gives a non-trivial extension of Mackey functors 
\[
0 \to \uB_{\ov{n}, \ov{1}^c} \to \uT(n) \to \uZ_{\ov{1}} \to 0. 
\]

\section{Computations for large $C_p$-fixed point dimensions}\label{largedimcomp}

This section deals with computations of $\uH^{\alpha}_{C_{p^n}}(S^0, \uZ)$ when the dimension of the $C_p$-fixed points of $\alpha$ is large. More precisely, we prove that if $|\alpha^{C_p}| \in (-2n+2, 2n)^c$, then the Mackey functor $\uH^{\alpha}_{C_{p^n}}(S^0;\uZ)$ can be computed explicitly in terms of  $|\alpha|$ and $\uH_{C_{p^{n-1}}}^{\alpha^{C_p}}(S^0; \uZ)$ (see Table \ref{comp-highfix}). We now drop $\uZ$ in the notation to write $\uH^{\alpha}_{C_{p^n}}(S^0)$ for $\uH^{\alpha}_{C_{p^n}}(S^0;\uZ)$, and throughout assume $n\geq 2$.

\begin{lemma}\label{pullbackalpha}
Let $\alpha \in \im(RO(C_{p^n}/C_{p^m}) \to RO(C_{p^n}))$. Then there is an equivalence 
\[\uH^\alpha_{C_{p^n}}(S^0) \cong \Phi^\ast_{p^m}(\uH^{\alpha^{C_{p^m}}}_{C_{p^n}/C_{p^m}}(S^0)).
\]
\end{lemma}
\begin{proof}
The hypothesis implies that $\alpha$ and hence, also $S^{-\alpha}$ are $C_{p^m}$-fixed. Suppose $q:C_{p^n} \to C_{p^n}/C_{p^m}$, which induces adjunctions \cite[Proposition V.3.10]{MM02} 
\[ 
\{ q^\ast(S^{-\alpha^{C_{p^m}}}), H\uZ\}^{C_{p^n}} \cong \{S^{-\alpha^{C_{p^m}}}, H\uZ^{C_{p^m}} \}^{C_{p^{n}}/C_{p^m}}. 
\] 
The result now follows from $H\uZ^{C_{p^m}} \simeq H\uZ$, and a calculation of the restriction and transfers as in the proof of \cite[Proposition 6.12]{BG20}.
%
%
%
 \end{proof}

Now apply Lemma \ref{pullbackalpha} together with the fact that there are invertible classes in degree $\lambda(rp^k)-\lambda(p^k)$ for $p\nmid r$. This implies that the conclusion is valid once $|\alpha^K|$ are all equal for $K\leq C_{p^m}$. Further applying the techniques of Example \ref{lambda0}, the hypothesis may be weakened to the fact $|\alpha^K|$ are all of the same sign for $K\leq C_{p^m}$.

\begin{prop}\label{signs}
Suppose $\alpha \in RO(C_{p^n})$ such that the  collection of numbers $|\alpha^{C_{p^k}}|$ for $k\leq m$, are either all  $0$ or are all of the same sign, then $\uH^\alpha_{C_{p^n}}(S^0) \cong \Phi^\ast_{p^m}(\uH^{\alpha^{C_{p^m}}}_{C_{p^n}/C_{p^m}}(S^0))$.
\end{prop}
%
%
We now readily deduce that in a range of degrees, the cohomology is $0$, starting from the fact that for the trivial group, the cohomology is $0$ if the degree is non-zero. 
\begin{prop}\label{cohzero} 
Let $\alpha \in RO(C_{p^n})$.\\
(a) If $|\alpha^{C_{p^m}}|$ positive for all $m$, or negative for all $m$, then, 
$
\uH^{\alpha}_{C_{p^n}}(S^0) \cong \uO.
$ \\ 
(b) If $|\alpha^{C_{p^m}}|\le 1$ for all $m,$ then,  $\uH^{\alpha}_{C_{p^n}}(S^0)\cong \uO.$\\
(c) For $\alpha$ odd satisfying  $|\alpha^{C_{p^m}}| < 1 \implies |\alpha^H| \le 1$ $\forall$ $C_{p^m} \subseteq H$, $\uH^{\alpha}_{C_{p^n}}(S^0 )\cong \uO.$ 
\end{prop}
\begin{proof}
Proposition \ref{signs} directly implies (a). For (b), the result follows from \eqref{anders-comp} and  (a) as,   $|\alpha|\le -1, |\alpha^H|\leq 1 $ $ \implies |(3-\lambda_0 - \alpha)^H| >0$ for all $H$. But for $|\alpha|= 1$, $|(3-\lambda_0-\alpha)|=0$, and all the other fixed point dimensions are $>0$. Using ~\myTagFormat{eq:base}{0}  along with (a) imply $\uH^{3-\lambda_0-\alpha}_{C_{p^n}}(S^0)\cong \uZ^\ast$, so that $\Ext_L$ is trivial. Hence (b)  follows.

For (c), Example \ref{lambda0} implies that $a_{\lambda_0}$ multiplication is an isomorphism if $|\alpha|\neq -1$ and is surjective if $|\alpha|=-1$. We now assume that (c) is true for $G=C_{p^k}$, $k< n$. For $|\alpha|>0$ and $|\alpha^{C_p}|>0$, Proposition \ref{signs} applies to prove the result. For $|\alpha|>0$ and $|\alpha^{C_p}|<0$, using the $a_{\lambda_0}$-multiplication, we get a surjective map from $\uH^\beta_{C_{p^n}}(S^0)$ with $|\beta|<0$ and $|\beta^{C_{p^l}}|=|\alpha^{C_{p^l}}|$ for $l>1$, so that Proposition \ref{signs} implies the result again. In the remaining cases, using the $a_{\lambda_0}$-multiplication, it suffices to prove for $|\alpha|=-1$, $|\alpha^{C_p}|=1$. The proof is now completed by repeated applications of Proposition \ref{signs} and Anderson duality\eqref{anders-comp} as 
\[
\uH^{\alpha}_{C_{p^n}}(S^0)  \cong \Ext_L(\uH^{3-\lambda_0 -\alpha}_{C_{p^n}}(S^0), \Z) \cong \Ext_L( \Phi_p^\ast \uH^{3-\alpha^{C_p}}_{C_{p^{n-1}}}(S^0),\Z) \cong \Phi_p^\ast \uH^{\alpha^{C_p} - \lambda_0}_{C_{p^{n-1}}}(S^0) \cong 0.
\]
\end{proof}

Proposition \ref{cohzero} has many applications in showing that certain even equivariant cell complexes have free cohomology. Examples include complex projective spaces and Grassmannians \cite{Lew88}, \cite{BG19}, and linear actions on simply connected $4$-manifolds \cite{BDK21}. 

\begin{prop}\label{comp_1}
Suppose $\alpha \in RO_0(C_{p^n})$ satisfying $|\alpha^{C_{p^k}}|>0$ for $k>1$. Then, 
\[\uH^\alpha_{C_{p^n}}(S^0)\cong  \begin{cases} \uZ_{\ov{1-\frac{|\alpha^{C_p}|}{2}}^c} & \text{ if } 4-2n \le |\alpha^{C_p}| \le 0\\ 
\uZ & \text{ if } |\alpha^{C_p}| \le 2-2n \\
 \uZ^\ast & \text{ if }|\alpha^{C_p}| \ge 2. \end{cases}
\]
\end{prop}
\begin{proof}
From the proof of Proposition \ref{cohzero} (b), observe that $\uH^\alpha_{C_{p^{n}}}(S^0) \cong \uZ^\ast$ for $|\alpha^{C_p}|>0$. We use induction on $F(\alpha)= 1-\frac{|\alpha^{C_p}|}{2}$. 
We now assume the result for $F(\alpha)<k$, and prove this first for those with $F(\alpha)=k\leq n$.   

Let $\beta\in RO_0(C_{p^n})$  with $|\beta^{C_{p^k}}|>0$ for $k>1$. Note that $F(\beta+\lambda_1 -\lambda_0) =F(\beta)-1$ and $\beta+\lambda_1-\lambda_0 \in RO_0(C_{p^n})$, so that it satisfies the induction hypothesis. Hence, using ~\myTagFormat{eq:base}{0} for index $\beta+\lambda_1-\lambda_0$, we obtain the long exact sequence
\[
\cdots \uH^{\beta+\lambda_1-1}_{C_{p^n}}(S(\lambda_0)_+)\to \uH^{\beta+\lambda_1-\lambda_0}_{C_{p^n}}(S^0) \to \uH^{\beta+\lambda_1}_{C_{p^n}}(S^0) \to \uH^{\beta+\lambda_1}_{C_{p^n}}(S(\lambda_0)_+) \cdots
\]

In the above, we have \eqref{lambda0coh} $\uH^{\beta+\lambda_1}_{C_{p^n}}(S(\lambda_0)_+)=0$ and $\uH^{\beta+\lambda_1-1}_{C_{p^n}}(S(\lambda_0)_+) =\uZ^\ast.$ The term $\uH^{\beta+\lambda_1-\lambda_0}_{C_{p^n}}(S^0) \cong \uZ_{\ov{F(\beta)-1}^c}$. Hence \eqref{nz_ext1}, $\uH^{\beta + \lambda_1}_{C_{p^n}}(S^0) \cong \uB_{\ov{n}, \ov{F(\beta)-1}^c}$.
Making use of ~\myTagFormat{eq:base}{1}, we obtain
\[
\xymatrix@R=0.3cm{0 \ar[r] & \uH^{\beta+\lambda_1-1}_{C_{p^n}}(S(\lambda_1)_+)\ar[r]  
\ar@{=}[d]^{(\text{by }  \ref{sph_coh})}  & \uH^{\beta}_{C_{p^n}}(S^0) \ar[r]  &\uH^{\beta+\lambda_1}_{C_{p^n}}(S^0)\ar@{=}[d] \ar[r] & \uH^{\beta+\lambda_1}_{C_{p^n}}(S(\lambda_1)_+)\ar@{=}^{(\text{by }  \ref{sph_coh})}[d]\cdots \\ & \uZ_{\ov{1}^c} &   &  \uB_{\ov{n},\ov{F(\beta)-1}^c}&  \uB_{\ov{1}, \emptyset}}
\]
The identifications at the two ends of the sequence are $\uH^{\beta+\lambda_1}_{C_{p^n}}(S(\lambda_1)_+)\cong \Theta_1(\uB_{\ov{1}, \emptyset})\cong \uB_{\ov{1}, \emptyset}$, and $\uH^{\beta+\lambda_1-1}_{C_{p^n}}(S(\lambda_1)_+) \cong \Theta^\ast_1(\uZ) \cong \uZ_{\ov{1}^c}.$ Proposition \ref{Bker} computes  the kernel of $\uH^{\beta+\lambda_1}_{C_{p^n}}(S^0) \to \uH^{\beta+\lambda_1}_{C_{p^n}}(S(\lambda_1)_+)$ as $\uB_{\ov{1}^c,\ov{F(\beta)}^c}.$  Hence we deduce the short exact sequence
\begin{myeq}\label{ind_step}
\uO \to \uZ_{\ov{1}^c}\to \uH^{\beta}_{C_{p^n}}(S^0) \to \uB_{\ov{1}^c,\ov{F(\beta)}^c} \to \uO.
\end{myeq}
Using ~\myTagFormat{eq:base}{0} and Corollary \ref{signs}, it follows $\uH^{\beta}_{C_{p^n}}(S^0) \to \uH^{\beta}_{C_{p^n}}(S(\lambda_0)_+)\cong \uZ$ is injective, hence $\uH^{\beta}_{C_{p^n}}(S^0)$ is torsion free. Thus \eqref{ind_step} along with Proposition \ref{zb} imply the middle term $\uH^{\beta}_{C_{p^n}}(S^0) \cong \uZ_{\ov{F(\beta)}^c}$. 

If $F(\alpha)=n$, then by the above computation $\uH^\alpha_{C_{p^n}}(S^0; \uZ)\cong \uZ.$ Working as above we get, $\uH^{\alpha+\lambda_0}_{C_{p^n}}(S^0)\cong \uB_{\ov{n}, \emptyset}$ (using ~\myTagFormat{eq:base}{0}),  and that the kernel $\uB_{\ov{n},\emptyset} \to \uB_{\ov{1},\emptyset}$ is $\uB_{\ov{1}^c,\emptyset}$. Analogously we obtain $\uH^{\alpha-\lambda_1 + \lambda_0}_{C_{p^n}}(S^0) \cong \uZ$. 
\end{proof}

\begin{table}[ht]

\begin{tabular}{ |p{3.8cm}|p{3.7cm}||p{3.8cm}| p{3.7cm}|  }
 \hline
{ \small $|\alpha|$ \text{ even}   } & { \small  $\uH_{C_{p^n}}^\alpha(S^0)$} &  { \small $|\alpha|$ \text{odd}  } & { \small $\uH_{C_{p^n}}^\alpha(S^0)$} \\
 \hline
{ \small $ |\alpha| >0$, $|\alpha^{C_p}|\leq -2n +2$}  & {\small $\uB_{\ov{n}, \emptyset} \oplus \Phi^\ast_p(\uH^{\alpha^{C_p}}_{C_{p^{n-1}}}(S^0) )$ } 
& { \small $|\alpha| <0 $, $|\alpha^{C_p}| \leq -2n+3$}  & { $\Phi^\ast_p(\uH^{\alpha^{C_p}}_{C_{p^{n-1}}}(S^0) )$ }  \\
\hline
{ \small $ |\alpha| =0$, $|\alpha^{C_p}|\leq -2n +2$}  & {\small $\uZ \oplus \Phi^\ast_p(\uH^{\alpha^{C_p}}_{C_{p^{n-1}}}(S^0) )$ }  & { \small $|\alpha| >0 $, $|\alpha^{C_p}| \leq -2n+3$}  & {\small $\Phi^\ast_p(\uH^{\alpha^{C_p}}_{C_{p^{n-1}}}(S^0) )$ }  \\
\hline
{ \small $ |\alpha| <0$, $|\alpha^{C_p}|\leq -2n +2$}  & {\small $\Phi^\ast_p(\uH^{\alpha^{C_p}}_{C_{p^{n-1}}}(S^0) )$ } 
& { \small $|\alpha| >0 $, $|\alpha^{C_p}| \geq 2n+1$}  & {\small $\Phi^\ast_p(\uH^{\alpha^{C_p}}_{C_{p^{n-1}}}(S^0) )$ } \\
\hline
{ \small $ |\alpha| \neq 0$, $|\alpha^{C_p}|\geq 2n $}  & {\small $ \Phi^\ast_p(\uH^{\alpha^{C_p}}_{C_{p^{n-1}}}(S^0) )$ }  & { \small $|\alpha| <0 $, $|\alpha^{C_p}| \geq 2n+1$}  & {\small $\uB_{\ov{n},\emptyset} \oplus \Phi^\ast_p(\uH^{\alpha^{C_p}}_{C_{p^{n-1}}}(S^0) )$ }  \\
 \hline
{ \small $ |\alpha| =0$, $|\alpha^{C_p}|\geq 2n $}  & {\small $\uZ^\ast \oplus \Phi^\ast_p(\uH^{\alpha^{C_p}}_{C_{p^{n-1}}}(S^0) )$ } 
&   &  \\
 \hline
\end{tabular}
\vspace{.2cm}
\caption{Formula for $\uH_{C_{p^n}}^\alpha(S^0)$ at large $C_p$-fixed points.}
\label{comp-highfix}
\end{table}

Incorporating the values obtained in Proposition \ref{comp_1} into the long exact sequence ~\myTagFormat{eq:base}{0}, we derive
\begin{cor}\label{comp_2}
Suppose $\alpha \in RO(C_{p^n})$ satisfying $|\alpha^H|>0$ for all $H \neq C_p.$ Then, 
\[\uH^\alpha_{C_{p^n}}(S^0; \uZ)\cong  \begin{cases} \uB_{\ov{n}, \ov{1-\frac{|\alpha^{C_p}|}{2}}^c} & \text{ if } 4-2n \le |\alpha^{C_p}| \le 0\\ \uB_{\ov{n}, \emptyset} & \text{ if } |\alpha^{C_p}| \le 2-2n \\ 0 & \text{ if }|\alpha^{C_p}| \ge 2. \end{cases}
\]
\end{cor}

We next use the multiplicative structure to derive a computation with $|\alpha|=0$ and $|\alpha^{C_p}|$ a sufficiently large negative number.
\begin{thm}\label{negfree}
If $\alpha \in RO_0(C_{p^n})$ with $|\alpha^{C_p}| \le -2n+2$, then the torsion-free part of $\uH^\alpha_{C_{p^n}}(S^0)$ is $\uZ$, and  there is a decomposition 
\[
\uH^\alpha_{C_{p^n}}(S^0) \cong \uZ \oplus \Phi^\ast_p(\uH^{\alpha^{C_p}}_{C_{p^{n-1}}}(S^0) ).
\]
For $|\alpha|<0$ and $|\alpha^{C_p}|\le -2n+2$ even, $\uH^\alpha_{C_{p^n}}(S^0)\cong \Phi^\ast_p(\uH^{\alpha^{C_p}}_{C_{p^{n-1}}}(S^0) ).$
\end{thm}
\begin{proof}
The last statement follows from Proposition \ref{signs}. In the rest of the proof, we use the $\uH^\bigstar_{C_{p^n}}(S^0)$-module structure on $\uH^\bigstar_{C_{p^n}}(S(\lambda_0)_+)$ which we denote by $\mu_{S(\lambda_0)}$. The multiplication in $\uH^\bigstar_{C_{p^n}}(S^0)$ is denoted by $\mu$. 

Proposition \ref{comp_1} implies that $\uH^\alpha_{C_{p^n}}(S^0) \cong \uZ$  for $\alpha \in RO_0(C_{p^n})$ satisfying $|\alpha^H|>0$ for all $H \neq e, C_p$ and $|\alpha^{C_p}|\le 2-2n$. Let $\beta\in RO_0(C_{p^n})$ satisfying $|\beta^{C_p}|=0$ and $|\beta^H|\le 0$ for all $H \neq e, C_p$. Applying Anderson duality \eqref{and_comp} and  Propositions \ref{comp_1}, \ref{cohzero} (a), we deduce $\uH^\beta_{C_{p^n}}(S^0; \uZ)\cong \uZ.$

We have the commutative diagram ($\pi : S(\lambda_0)_+ \to S^0$ is the map induced by adding disjoint base-points to  $S(\lambda_0)\to \ast$ ) 
\[ 
\xymatrix{\uZ \cong \uH^\alpha_{C_{p^n}}(S^0) \square_{\uZ}\uH^\beta_{C_{p^n}}(S^0)\ar[d]^{\pi^\ast \square_{\uZ}\uZ} \ar[rr]^-{\mu} && \uH^{\alpha+\beta}_{C_{p^n}}(S^0) \ar[d]^{\pi^\ast}\\ 
\uZ \cong  \uH^\alpha_{C_{p^n}}(S(\lambda_0)_+) \square_{\uZ}\uH^\beta_{C_{p^n}}(S^0) \ar[rr]^-{\mu_{S(\lambda_0)}} && \uH^{\alpha+\beta}_{C_{p^n}}(S(\lambda_0)_+)\cong \uZ}
\]
 Note that $\pi^\ast \square_{\uZ}\uZ$ and  $\mu_{S(\lambda_0)}$ are isomorphisms, so that $\mu$ is a section for $\pi^\ast$ up to isomorphism. This yields the decomposition (using ~\myTagFormat{eq:base}{0})
 \[
 \uH^{\alpha+\beta}_{C_{p^n}}(S^0) \cong \uZ \oplus \uH^{\alpha+\beta-\lambda_0}_{C_{p^n}}(S^0, \uZ)\cong \uZ \oplus \Phi^\ast_p(\uH^{(\alpha+\beta)^{C_p}}_{C_{p^{n-1}}}(S^0, \uZ)) (\text{by Corollary } \ref{signs}).
 \]
 Hence the result follows as for each $H \neq e, C_p,$ $|(\alpha+\beta)^H|$ can be made arbitrary.
\end{proof}
For $\alpha$ odd positive case, the following is a direct consequence of Theorem \ref{negfree}.
\begin{cor}\label{negtor}
Suppose $\alpha\in RO(C_{p^n})$ satisfying $|\alpha|>0$ odd and $|\alpha^{C_p}| \le -2n+3,$ then $\uH^{\alpha}_{C_{p^n}}(S^0) \cong \Phi^\ast_p(\uH^{\alpha^{C_p}}_{C_{p^{n-1}}}(S^0))$. 
\end{cor}
\begin{proof}
From Example \ref{lambda0}, it suffices to prove for $|\alpha|=1$. 
Then by  ~\myTagFormat{eq:base}{0}, one obtains
\[
\xymatrix@R=0.5cm@C=0.5cm{\cdots \uH^{\alpha-1}_{C_{p^n}}(S^0) \ar@{=}[d]^{\text{(by Theorem } \ref{negfree})}\ar[r] 
&\uH^{\alpha-1}_{C_{p^n}}(S(\lambda_0)_+) \ar@{=}[d]\ar[r] 
& \uH^{\alpha-\lambda_0}_{C_{p^n}}(S^0) \ar[r] & \uH^{\alpha}_{C_{p^n}}(S^0) \ar[r] 
& \uH^{\alpha}_{C_{p^n}}(S(\lambda_0)_+)\ar@{=}[d] \cdots \\ 
\uZ \oplus \Phi^\ast_p(\uH^{\alpha^{C_p}-1}_{C_{p^{n-1}}}(S^0, \uZ) )  
& \uZ & & 
& \uZ^\ast}
\]
The map $\uH^{\alpha-1}_{C_{p^n}}(S^0)\to \uH^{\alpha-1}_{C_{p^n}}(S(\lambda_0)_+)$ has a section according to the proof of Theorem \ref{negfree}. Thus, there is an isomorphism
$
\uH^{\alpha-\lambda_0}_{C_{p^n}}(S^0) \stackrel{\cong}{\to} \uH^{\alpha}_{C_{p^n}}(S^0).
$
The result follows from  $\uH^{\alpha-\lambda_0}_{C_{p^n}}(S^0)\cong \Phi^\ast_p(\uH^{\alpha^{C_p}}_{C_{p^{n-1}}}(S^0))$ using  Corollary \ref{signs}. 
\end{proof}

The following corollary is a direct consequence of the Anderson duality and Theorem \ref{negfree}.
\begin{cor}\label{zerpos}
The Mackey functor $\uH^{\alpha}_{C_{p^n}}(S^0)\cong \uZ^* \oplus \Phi^\ast_p(\uH^{\alpha^{C_p}}_{C_{p^{n-1}}}(S^0))$ if $|\alpha|=0$ and $|\alpha^{C_p}|\ge 2n.$
\end{cor}

\begin{proof}
Using Theorem \ref{negfree} and Corollary \ref{negtor}, the short exact sequence \eqref{and_comp} turns out to be 
\[
\uO \to  \Phi^\ast_p(\uH^{3-\alpha^{C_p}}_{C_{p^{n-1}}}(S^0))^E \to \uH^{\alpha}_{C_{p^n}}(S^0) \to \uZ^\ast \to \uO.
\]
The map $\uZ^\ast \cong \uH^{\alpha+\lambda_0 -1 }_{C_{p^n}}(S^0)\to \uH^\alpha_{C_{p^n}}(S^0)$ in ~\myTagFormat{eq:base}{0} serves as a splitting for this sequence. Applying \eqref{anders-comp} to $\uH^{3-\alpha^{C_p}}_{C_{p^{n-1}}}(S^0)$, we obtain the result using the identifications in Example \ref{lambda0}. 
%
\end{proof}

The following result is readily deduced from Corollary \ref{negtor} and Anderson duality \eqref{anders-comp}
\begin{cor}\label{postor}
Suppose $\alpha\in RO(C_{p^n})$ satisfying $|\alpha|<0$ even and $|\alpha^{C_p}| \ge 2n,$ then $\uH^{\alpha}_{C_{p^n}}(S^0) \cong \Phi^\ast_p(\uH^{\alpha^{C_p}}_{C_{p^{n-1}}}(S^0))$. 
\end{cor}

The following result is obtained by applying ~\myTagFormat{eq:base}{0} to the results of Theorem \ref{negfree}.
\begin{cor}\label{posevtor}
Suppose $\alpha \in RO(C_{p^n})$ satisfying $|\alpha|>0$ even and $|\alpha^{C_p}|\le -2n+2$, then $\uH^{\alpha}_{C_{p^n}}(S^0) \cong \uB_{\ov{n}, \emptyset} \oplus \Phi^\ast_p(\uH^{\alpha^{C_p}}_{C_{p^{n-1}}}(S^0) ).$ \end{cor}

\begin{proof}
From Example \ref{lambda0}, it suffices to assume $|\alpha|=2$. We then have the following reduction of ~\myTagFormat{eq:base}{0}
\[
\uO \to \uZ^* \to \uZ \oplus \Phi^\ast_p(\uH^{\alpha^{C_p}}_{C_{p^{n-1}}}(S^0) ) \to \uH^{\alpha}_{C_{p^n}}(S^0) \to \uO.
\]
As $\uZ^\ast$ is generated by transfers on the element $1 \in \Z = \uZ^\ast(C_{p^n}/e)$, the image of $\uZ^* \to \uZ \oplus \Phi^\ast_p(\uH^{\alpha^{C_p}}_{C_{p^{n-1}}}(S^0) )$ sits inside $\uZ$, uniquely defined by the fact that it is an isomorphism at $C_{p^n}/e$. The result follows as the cokernel of $\uZ^\ast \to \uZ$ is $\uB_{\ov{n},\emptyset}$. 
\end{proof}

For the group $C_p$, the equivariant cohomology ring (in grading $n+m\lambda_0$) consists of a polynomial algebra $\Z[u_{\lambda_0}, a_{\lambda_0}]$, plus modules $\Z\{[pu_{\lambda_0}^{-j}]\}$ and $\Z/p\{u_{\lambda_0}^{-j} a_{\lambda_0}^{-k}\}$. The latter part are all $a_{\lambda_0}$-torsion, and the former part gives a region where multiplication by $a_{\lambda_0}$ is injective, but we do not have any $a_{\lambda_0}$-periodic pieces.  For $n\geq 2$, the Mackey functor $ \Phi^\ast_p(\uH^{\alpha^{C_p}}_{C_{p^{n-1}}}(S^0))$ forms a $a_{\lambda_0}$-periodic piece as observed in the proofs of Theorem \ref{negfree} and  Corollary \ref{posevtor}, and in Corollary \ref{zerpos} and Corollary \ref{postor}. Using Lemma \ref{pullbackalpha}, this also constructs $a_{\lambda_i}$-periodic pieces over $C_{p^n}$ whenever $n\geq i+2$. 

\begin{exam}
For the group $C_{p^2}$, an example where $a_{\lambda_0}$-periodicity occurs is in degrees $2\lambda_1 + c \lambda_0$, where the periodic piece is the Mackey functor $\uB_{\{2\},\emptyset}$ which is $\Z/p$ at $C_{p^2}/C_{p^2}$, and $0$ at other levels. The Mackey functor in degree $2\lambda_1 - 2\lambda_0$ is demonstrated in Table \ref{per0}, where the symbols for the generating classes are written alongside. The class $[p^2 u_{\lambda_0}^{-2}]$ is $a_{\lambda_0}$-torsion, and the periodic piece is represented in degree $2\lambda_1$ by $a_{\lambda_0}^2 a_{\lambda_1/\lambda_0}^2 = a_{\lambda_1}^2$. This construction easily generalizes so that over $C_{p^n}$ for $n\geq 2$, the class $a_{\lambda_1}^n$ is $a_{\lambda_0}$-periodic, and for $n\geq i+2$, $a_{\lambda_{i+1}}^{n-i}$ is $a_{\lambda_i}$-periodic.  

\begin{table}[ht]

\begin{tabular}{ |p{4.5cm}|p{1.7cm}|  }
 \hline
{  \text{Mackey functor diagram}} &  { \text{ \;\; \;\; Generating elements}}  \\
 \hline
{\xymatrix{  \Z \oplus \Z/p \ar@/_1pc/[d]_{\small \begin{bmatrix} p & 0\end{bmatrix}}  \\
  \Z \ar@/_1pc/[d]_p \ar@/_1pc/[u]_{\small \begin{bmatrix} 1 \\ 0 \end{bmatrix}} \\
     \Z  \ar@/_1pc/[u]_1} } & { \xymatrix@R=0.70cm{  u_{\lambda_1}^2[p^2u_{\lambda_0}^{-2}], (a_{\lambda_1/\lambda_0})^2 -u_{\lambda_1}^2[p^2u_{\lambda_0}^{-2}] \\
      [pu_{\lambda_0}^{-2}] \\
       1}} 
 \\
 \hline
\end{tabular}
\vspace{.2cm}
\caption{Formula for $\uH_{C_{p^2}}^{2\lambda_1-2\lambda_0} (S^0)$.}
\label{per0}
\end{table}

\end{exam}

Note that $\Ext_L(\uB_{\ov{n}, \emptyset}, \Z) \cong \uB_{\ov{n}, \emptyset}$, hence \eqref{anders-comp} along with Corollary \ref{posevtor} readily implies
\begin{cor}\label{posoddtor}
For an element $\alpha \in RO(C_{p^n})$ satisfying $|\alpha|<0$ odd and $|\alpha^{C_p}| \ge 2n+1$, $\uH^{\alpha}_{C_{p^n}}(S^0) \cong \underline{B}_{\ov{n}, \emptyset} \oplus \Phi^\ast_p(\uH^{\alpha^{C_p}}_{C_{p^{n-1}}}(S^0, \uZ) ).$
\end{cor}

\section{Examples of non-trivial extensions}\label{nontriv}
 In this section, we compute the coefficient Mackey functor $\uH^{\bigstar}_{C_{p^2}}(S^0; \uZ)$ completely, and using that we observe non-trivial extensions for the group $C_{p^n}$ when $n\ge 3$. The ring structure of $\uH^\bigstar_{C_{p^2}}(S^0; \uZ)(C_{p^2}/C_{p^2})$ is completely determined in \cite{Zen17} using the Tate square. Our computations follow mainly from the general results of \S \ref{largedimcomp}. 
%
For $C_p$, the $\uZ$-cohomology has the following additive structure \cite[Corollary B.10]{Fer99}

\begin{myeq}\label{cohCp}
\uH^\alpha_{C_p}(S^0;\uZ) \cong \begin{cases}
\uZ  & \mbox{if}~|\alpha|=0, |\alpha^{C_p}|\leq 0 \\
\uZ^\ast  & \mbox{if}~|\alpha|=0, |\alpha^{C_p}|> 0\\
\uB_{\ov{1}, \emptyset}  & \mbox{if}~|\alpha|>0, |\alpha^{C_p}|\leq 0~\mbox{even} \\
\uB_{\ov{1}, \emptyset} & \mbox{if}~|\alpha|<0, |\alpha^{C_p}|\ge 3 ~\mbox{odd} \\
0  &\mbox{otherwise}.
\end{cases}
\end{myeq}

\begin{mysubsection}{The additive structure of $C_{p^2}$ cohomology}
We compute the Mackey functors $\uH^\alpha_{C_{p^2}}(S^0; \uZ)$ for each $\alpha \in RO(C_{p^2})$. The entire computation is summarized in Table \ref{comp-tab}. 
\end{mysubsection}

\begin{table}[ht]

\begin{tabular}{ |p{3.5cm}|p{1.7cm}||p{3.5cm}| p{1.8cm}|  }
 \hline
{ \small $|\alpha|$ \text{ even \& positive}} & { \small  $\uH_{C_{p^2}}^\alpha(S^0)$} &  { \small $|\alpha|=0$ } & { \small $\uH_{C_{p^2}}^\alpha(S^0)$} \\
 \hline
{ \small $ |\alpha^{C_p}| >0 , |\alpha^{C_{p^2}}| >0$}  & { $0$ } 
& { \small $|\alpha^{C_p}| \le 0 , |\alpha^{C_{p^2}}| \le 0$}  & { $\uZ$ }  \\
\hline
{\small $|\alpha^{C_p}| >0,|\alpha^{C_{p^2}}|\le 0$}  &  { $\uB_{\ov{1}^c, \emptyset}$ } &  {\small $|\alpha^{C_p}| >0,|\alpha^{C_{p^2}}|> 0$}  &  { $\uZ^{\ast}$ } \\
\hline
{ \small  $|\alpha^{C_p}|=0,|\alpha^{C_{p^2}}|\le 0$}  &  { $\uB_{\ov{2}, \emptyset}$ }
& { \small  $|\alpha^{C_p}|\ge 4,|\alpha^{C_{p^2}}|\le 0$}  &  { $\uB_{\ov{1}^c, \emptyset}\oplus \uZ^\ast$ } \\
\hline
 { \small  $|\alpha^{C_p}|= 0, |\alpha^{C_{p^2}}| >0$}  &  { $\uB_{\ov{2}, \ov{1}^c}$ } & { \small  $|\alpha^{C_p}|= 2, |\alpha^{C_{p^2}}| \le 0$}  &  { $\uZ_{\ov{1}}$ } \\
 \hline
{ \small  $|\alpha^{C_p}|<0, |\alpha^{C_{p^2}}|>0$}  &  { $\uB_{\ov{2}, \emptyset}$ }
& { \small  $|\alpha^{C_p}|=0, |\alpha^{C_{p^2}}|>0$}  &  { $\uZ_{\ov{1}^c}$ }\\
 \hline
 { \small  $|\alpha^{C_p}|\le 0, |\alpha^{C_{p^2}}|\le 0$}  &  { $\uB_{\ov{2}, \emptyset}$} & { \small  $|\alpha^{C_p}|<0, |\alpha^{C_{p^2}}|> 0$}  &  { $\uZ$ } \\
\Cline{1.3pt}{1-4}
 \hline
 
{ \small $|\alpha|$ \text{ odd \& negative} } & { \small  $\uH_{C_{p^2}}^\alpha(S^0)$} & {\small $|\alpha|$ \text{ even \& negative}  } & { \small $\uH_{C_{p^2}}^\alpha(S^0)$} \\
\hline
{ \small $ |\alpha^{C_p}|\le 1, |\alpha^{C_{p^2}}|\leq 1$}  & { $0$ } 
&{\small $|\alpha^{C_p}| \ge 4,|\alpha^{C_{p^2}}| \le 0$}  &  { $\uB_{\ov{1}^c, \emptyset}$ }  \\
\hline
{\small $|\alpha^{C_p}| \le 1,|\alpha^{C_{p^2}}|> 1$}  &  { $\uB_{\ov{1}^c, \emptyset}$ }  & { \small  otherwise }  &  { 0 } \\
\hline 
\Cline{1pt}{3-4}
{ \small  $|\alpha^{C_p}|=3,|\alpha^{C_{p^2}}|> 1$}  &  { $\uB_{\ov{2}, \emptyset}$ }
&{\small $|\alpha|$ \text{ odd \& positive}  } & { \small $\uH_{C_{p^2}}^\alpha(S^0)$} \\
\hline
{ \small  $|\alpha^{C_p}|= 3, |\alpha^{C_{p^2}}| \le 1$}  &  { $\uB_{\ov{1}, \emptyset}$ } & {\small $|\alpha^{C_p}|\le -1, |\alpha^{C_{p^2}}| >1$}  & { $\uB_{\ov{1}^c, \emptyset}$ }  \\
\hline
{ \small  $|\alpha^{C_p}|\ge 5$}  &  { $\uB_{\ov{2}, \emptyset}$ }
& {\small otherwise}  & $0$ \\
\hline
\end{tabular}
\vspace{.2cm}
\caption{Formula for $\uH_{C_{p^2}}^\alpha(S^0;\uZ)$.}
\label{comp-tab}
\end{table}

\begin{theorem}\label{Cp2-comp}
The Mackey functors $\uH_{C_{p^2}}^\alpha(S^0;\uZ)$ are as demonstrated in Table \ref{comp-tab}.
\end{theorem}
\begin{proof}
The starting point is the application Proposition \ref{signs} to  \eqref{cohCp}, which gives the result whenever $|\alpha|$ and $|\alpha^{C_p}|$ are either both $0$, or both have the same sign. The remaining cases with $|\alpha|=0$ follow from Proposition \ref{comp_1}, Theorem \ref{negfree}, and Corollary \ref{zerpos}. For $|\alpha|>0$ odd, apart from the above, the remaining follow from Corollary \ref{negtor}. Applying Anderson duality \eqref{anders-comp} to the cases with $|\alpha|>0$ odd, we obtain the computations for $|\alpha|<0$ even.     

We next consider the case $|\alpha|>0$ even. Summarizing the results from \S \ref{largedimcomp}, we note that other than $|\alpha^{C_p}|=0$, $|\alpha^{C_{p^2}}|\leq 0$, the results follow from Proposition \ref{signs}, Corollary \ref{comp_2} and Corollary \ref{posevtor}. Using the calculations in Example \ref{lambda0},  it suffices to consider $|\alpha|=2$, in which case we already have computed the cohomology at the grading $\alpha-\lambda_0$ to be $\uZ$ \eqref{comp-tab}. The short exact sequence
 \[
 \uO \to \uZ^\ast \to  \uH^{\alpha-\lambda_0}_{C_{p^2}}(S^0; \uZ) \to  \uH^{\alpha}_{C_{p^2}}(S^0; \uZ) \to \uO. 
 \]
computes  $\uH^\alpha_{C_{p^2}}(S^0)\cong \uB_{\ov{2}, \emptyset}$.  Finally for $|\alpha|<0$ odd, the result follows from Anderson duality \eqref{anders-comp}.
\end{proof}

\begin{mysubsection}{Examples of non-trivial extensions}
We now point out cohomology computations where the Mackey functors which arise are not a direct sum of copies of $\uZ_T$ and $\uB_{T,S}$. 
%
%
We start by assuming $n\ge 3$ and  $\alpha \in RO_0(C_{p^n})$ satisfying $|\alpha^{C_p}|=2n-2$ and $|\alpha^H|\le 0$ for all $H \neq e, C_p.$ By Proposition \ref{sph_coh}, there is a short exact sequence
\begin{myeq}\label{nonsplit}
\uO \to \uB_{\ov{n},\ov{1}^c} \cong \Theta^\ast_1(\uB_{\ov{1}}) \to \uH^\alpha_{C_{p^n}}(S(\lambda_1)_+; \uZ) \to \Theta_1(\uZ_{\ov{1}}) \cong \uZ_{\ov{1}}\to \uO
\end{myeq}
in $\Mack_{C_{p^n}}$. At each level $C_{p^n}/H$, the short exact sequence \eqref{nonsplit} splits as the right end is free. We note that there does not exists any splitting in $\Mack_{C_{p^n}}$. For if it did, then $\downarrow^{p^n}_{p^2} \uH^\alpha_{C_{p^n}}(S(\lambda_1)_+; \uZ) \cong  \uB_{\ov{2},\ov{1}^c} \oplus \uZ_{\ov{1}}$. Now we may apply ~\myTagFormat{eq:base}{1} substituting the values from Table \ref{comp-tab} to obtain the following  exact sequence
\[
\uO \to \uB_{\ov{1}^c,\emptyset} \to  \uZ^\ast \oplus \uB_{\ov{1}^c, \emptyset }\to \uB_{\ov{2},\ov{1}^c} \oplus \uZ_{\ov{1}} \to \uH^{\alpha-\lambda_1+1}_{C_{p^2}}(S^0)\to \uO.
\]
The restriction $\res^{C_{p^2}}_{C_p}$ in $\uB_{\ov{2},\ov{1}^c}$ is $0$ and  in $\uZ_{\ov{1}}$ is divisible by $p$, on the other hand, the restriction in $\uH^{\alpha-\lambda_1+1}_{C_{p^2}}(S^0)\cong \uB_{\ov{2},\emptyset}$ (Table \ref{comp-tab}) is the usual quotient $\Z/p^2 \to \Z/p$. Thus, the short exact sequence \eqref{nonsplit} does not split in $\Mack_{C_{p^n}}.$ 

Assuming \eqref{nonsplit} does not split, a direct computation implies that up to isomorphism, $\uH^\alpha_{C_{p^n}}(S(\lambda_1)_+)$ has only one choice which we call $ \uT(n)$. The groups are given by 
\[
\uT(n)(C_{p^n}/H) \cong \begin{cases} \Z \oplus \Z/p & \mbox{if } H\neq e \\ 
                                                              \Z  & \mbox{if } H=e, \end{cases} 
\]
 and the restrictions and transfers are given by
\[\res^{C_{p^i}}_{C_{p^{i-1}}}= \begin{cases} \begin{pmatrix} 1 & 0 \\ 0 & 0 \end{pmatrix} & \mbox{ for } 2\le i \le n \\  \begin{pmatrix} p & 0\end{pmatrix} & \mbox{ for } i=1. \end{cases}  \text{ and } \tr^{C_{p^i}}_{C_{p^{i-1}}}= \begin{cases} \begin{pmatrix} p & 0 \\ 0 & 1 \end{pmatrix} & \mbox{ for } 2\le i \le n \\  \begin{pmatrix} 1\\ -1\end{pmatrix} & \mbox{ for } i=1. \end{cases} 
\]

Now, we restrict out attention to the group $C_{p^3}$ and $\alpha \in RO_0(C_{p^3})$ satisfying $|\alpha^{C_p}|=4$, $|\alpha^{C_{p^2}}|\le 0$, and $|\alpha^{C_{p^3}}|\le 0$. Consider the long exact sequence 
\begin{myeq}\label{p-seq3}
\cdots \uH^{\alpha-\lambda_1}_{C_{p^3}}(S^0) \to \uH^{\alpha}_{C_{p^3}}(S^0) \to \uH^{\alpha}_{C_{p^3}}(S(\lambda_1)_+) \to \uH^{\alpha-\lambda_1+1}_{C_{p^3}}(S^0) \to \uH^{\alpha+1}_{C_{p^3}}(S^0)\cdots
\end{myeq}
From \eqref{anders-comp} and Proposition \ref{cohzero} (a), we get $\uH^{\alpha-\lambda_1}_{C_{p^3}}(S^0)\cong \uO.$ Using \eqref{anders-comp}, we obtain $\uH^{\alpha-\lambda_1+1}_{C_{p^3}}(S^0) \cong \Ext_L(\uH^{2-\lambda_0+\lambda_1-\alpha}_{C_{p^3}}(S^0), \Z)$, which is $\uB_{\ov{3},\ov{1}^c}^E \cong \uB_{\ov{1}, \emptyset}$ by Corollary \ref{comp_2}. Finally, $\uH^{\alpha+1}_{C_{p^3}}(S^0)=0$ by Proposition \ref{cohzero} (b). Thus \eqref{p-seq3} reduces to the short exact sequence
\[
\uO \to \uH^{\alpha}_{C_{p^3}}(S^0) \to \uT(3) \to \uB_{\ov{1}, \emptyset} \to \uO.
\]
A direct computation of the kernel of the map $\uT(3) \to \uB_{\ov{1}, \emptyset}$ gives 
\[
\xymatrix{\Z \oplus \Z/p \ar@/_1pc/[d]_{\small \begin{bmatrix} 1 & 0\\ 0 & 0 \end{bmatrix}} \\
 \Z \oplus \Z/p \ar@/_1pc/[d]_{\small \begin{bmatrix} p & 0\end{bmatrix}} \ar@/_1pc/[u]_{\small \begin{bmatrix} p & 0\\ 0 & 1 \end{bmatrix}} \\
  \Z \ar@/_1pc/[d]_p \ar@/_1pc/[u]_{\small \begin{bmatrix} 1 \\ -1 \end{bmatrix}} \\
   \Z  \ar@/_1pc/[u]_1}
\]
Applying a change of basis $\epsilon_1 \mapsto \epsilon_1-\epsilon _2$ and $\epsilon_2\mapsto \epsilon_2$ at $C_{p^3}/C_{p^2}$, we may rewrite the Mackey functor above as 
\[
\xymatrix{& \Z \oplus \Z/p \ar@/_1pc/[d]_{\small \begin{bmatrix} 1 & 0\\ 1 & 0 \end{bmatrix}} \\ 
\uH^{\alpha}_{C_{p^3}}(S^0): & \Z \oplus \Z/p \ar@/_1pc/[d]_{\small \begin{bmatrix} p & 0\end{bmatrix}} \ar@/_1pc/[u]_{\small \begin{bmatrix} p & 0\\ -1 & 1 \end{bmatrix}} \\
 & \Z \ar@/_1pc/[d]_p \ar@/_1pc/[u]_{\small \begin{bmatrix} 1 \\ 0 \end{bmatrix}} \\ &  \Z  \ar@/_1pc/[u]_1}
\]
Note that $\downarrow^{p^3}_{p^2}\uH^{\alpha}_{C_{p^3}}(S^0) \cong \uZ^\ast \oplus \uB_{\ov{1}^c, \emptyset}$ as in Table \ref{comp-tab}. We now look at \eqref{p-seq3}  
\[ 
\uO \to \uH^{\alpha+\lambda_0-1}_{C_{p^3}}(S(\lambda_0)_+)  \to  \uH^{\alpha}_{C_{p^3}}(S^0) \to \uH^{\alpha+\lambda_0}_{C_{p^3}}(S^0) \to \uH^{\alpha+\lambda_0}_{C_{p^3}}(S(\lambda_0)_+) \cong \uO.
\]
and put in the values to get 
\[
\xymatrix{&  \Z \ar@/_/[d]_p \ar[rrr]^{\small \begin{bmatrix} p \\ -1 \end{bmatrix}}&& & \Z \oplus \Z/p \ar@/_1pc/[d]_{\small \begin{bmatrix} 1 & 0\\ 1 & 0 \end{bmatrix}} \ar[rrr] &&& \Z/p^2  \ar@/_1pc/[d]_{1}\\ 
\uO \ar[r] & \Z \ar@/_/[d]_p \ar@/_/[u]_1  \ar[rrr]^(0.3){\small \begin{bmatrix} 1 \\ 0 \end{bmatrix}} &  &&  \Z \oplus \Z/p \ar@/_1pc/[d]_{\small \begin{bmatrix} p & 0\end{bmatrix}} \ar@/_1pc/[u]_{\small \begin{bmatrix} p & 0\\ -1 & 1 \end{bmatrix}}   \ar[rrr]&&& \Z/p \ar@/_1pc/[u]_{p} \ar@/_/[d] \ar[r] & \uO,\\  
& \Z \ar@/_/[d]_p\ar@/_/[u]_1 \ar[rrr]^{1} &&&   \Z \ar@/_/[d]_p \ar@/_1pc/[u]_{\small \begin{bmatrix} 1 \\ 0 \end{bmatrix}} \ar[rrr] &&& 0 \ar@/_/[d] \ar@/_/[u] \\ 
& \Z \ar@/_/[u]_1 \ar[rrr]^{1} &&& \Z  \ar@/_/[u]_1 \ar[rrr] &&& 0 \ar@/_/[u]} 
 \]
 a non-trivial extension. One may compute and check that $\uH^\alpha_{C_{p^3}}(S^0)$ does not split as a direct sum of Mackey functors of the type $\uZ_T$ and $\uB_{T,S}$. 
 
\begin{table}[ht]

\begin{tabular}{ |p{4.5cm}|p{1.7cm}|  }
 \hline
{  \text{ Mackey functor diagram}} &  { \text{     \;\; \;\;   Generating elements}}  \\
 \hline
{\xymatrix{ \Z \oplus \Z/p \ar@/_1pc/[d]_{\small \begin{bmatrix} 1 & 0\\ 1 & 0 \end{bmatrix}} \\
  \Z \oplus \Z/p \ar@/_1pc/[d]_{\small \begin{bmatrix} p & 0\end{bmatrix}} \ar@/_1pc/[u]_{\small \begin{bmatrix} p & 0\\ -1 & 1 \end{bmatrix}} \\  \Z \ar@/_1pc/[d]_p \ar@/_1pc/[u]_{\small \begin{bmatrix} 1 \\ 0 \end{bmatrix}} \\
     \Z  \ar@/_1pc/[u]_1} } & { \xymatrix@R=.65cm{ (a_{\lambda_1/\lambda_0})^2, p.(a_{\lambda_1/\lambda_0})^2 -u_{\lambda_1}^2[p^3u_{\lambda_0}^{-2}]\\ 
      u_{\lambda_1}^2[p^2u_{\lambda_0}^{-2}], (a_{\lambda_1/\lambda_0})^2 -u_{\lambda_1}^2[p^2u_{\lambda_0}^{-2}] \\
       [pu_{\lambda_0}^{-2}] \\ 
       1}} 
 \\
 \hline
\end{tabular}
\vspace{.2cm}
\caption{Formula for $\uH_{C_{p^3}}^{2\lambda_1 - 2 \lambda_0} (S^0)$.}
\label{comp-cp3}
\end{table}

\begin{exam}
An example of $\alpha$ as the above is $2\lambda_1 - 2\lambda_0$. The  diagram (Table \eqref{comp-cp3}) compares the Mackey functor diagram with the corresponding generating classes.

\end{exam}

\end{mysubsection}


\begin{thebibliography}{10}

\bibitem{And69}
{\sc D.~W. Anderson}, {\em Universal coefficient theorems for k-theory}.
\newblock 1969.

\bibitem{BDK21}
{\sc S.~Basu, P.~Dey, and A.~Karmakar}, {\em Equivariant homology
  decompositions for definite 4-manifolds}, 2021.

\bibitem{BG19}
{\sc S.~Basu and S.~Ghosh}, {\em Computations in {$C_{pq}$}-{B}redon
  cohomology}, Math. Z., 293 (2019), pp.~1443--1487.

\bibitem{BG20}
\leavevmode\vrule height 2pt depth -1.6pt width 23pt, {\em Equivariant
  cohomology for cyclic groups of square-free order}, 2020.
\newblock available at \url{https://arxiv.org/abs/2006.09669}.

\bibitem{Fer99}
{\sc K.~K. Ferland}, {\em On the {RO}({G})-graded equivariant ordinary
  cohomology of generalized {G}-cell complexes for {G} = {Z}/p}, ProQuest LLC,
  Ann Arbor, MI, 1999.
\newblock Thesis (Ph.D.)--Syracuse University.

\bibitem{GM95}
{\sc J.~P.~C. Greenlees and J.~P. May}, {\em Equivariant stable homotopy
  theory}, in Handbook of algebraic topology, North-Holland, Amsterdam, 1995,
  pp.~277--323.

\bibitem{HHR16}
{\sc M.~A. Hill, M.~J. Hopkins, and D.~C. Ravenel}, {\em On the nonexistence of
  elements of {K}ervaire invariant one}, Ann. of Math. (2), 184 (2016),
  pp.~1--262.

\bibitem{HHR17}
{\sc M.~A. Hill, M.~J. Hopkins, and D.~C. Ravenel}, {\em The slice spectral
  sequence for certain {$RO(C_{p^n})$}-graded suspensions of {$H\underline{\bf
  Z}$}}, Bol. Soc. Mat. Mex. (3), 23 (2017), pp.~289--317.

\bibitem{KL20}
{\sc I.~Kriz and Y.~Lu}, {\em On the {$RO(G)$}-graded coefficients of dihedral
  equivariant cohomology}, Math. Res. Lett., 27 (2020), pp.~1109--1128.

\bibitem{Lew88}
{\sc L.~G. Lewis, Jr.}, {\em The {$R{\rm O}(G)$}-graded equivariant ordinary
  cohomology of complex projective spaces with linear {${\bf Z}/p$} actions},
  in Algebraic topology and transformation groups ({G}\"ottingen, 1987),
  vol.~1361 of Lecture Notes in Math., Springer, Berlin, 1988, pp.~53--122.

\bibitem{MM02}
{\sc M.~A. Mandell and J.~P. May}, {\em Equivariant orthogonal spectra and
  {$S$}-modules}, Mem. Amer. Math. Soc., 159 (2002), pp.~x+108.

\bibitem{May96}
{\sc J.~P. May}, {\em Equivariant homotopy and cohomology theory}, vol.~91 of
  CBMS Regional Conference Series in Mathematics, Published for the Conference
  Board of the Mathematical Sciences, Washington, DC; by the American
  Mathematical Society, Providence, RI, 1996.
\newblock With contributions by M. Cole, G. Comeza\~na, S. Costenoble, A. D.
  Elmendorf, J. P. C. Greenlees, L. G. Lewis, Jr., R. J. Piacenza, G.
  Triantafillou, and S. Waner.

\bibitem{Web00}
{\sc P.~Webb}, {\em A guide to {M}ackey functors}, in Handbook of algebra,
  {V}ol. 2, vol.~2 of Handb. Algebr., Elsevier/North-Holland, Amsterdam, 2000,
  pp.~805--836.

\bibitem{Yos83}
{\sc T.~Yoshida}, {\em On {$G$}-functors. {II}. {H}ecke operators and
  {$G$}-functors}, J. Math. Soc. Japan, 35 (1983), pp.~179--190.

\bibitem{Zen17}
{\sc M.~Zeng}, {\em Equivariant {E}ilenberg-{M}ac{L}ane spectra in cyclic
  $p$-groups}, 2017.
\newblock available at \url{https://arxiv.org/abs/1710.01769}.

\end{thebibliography}
\end{document}